\documentclass[12pt,a4paper]{scrartcl}
\usepackage[english]{babel}

\usepackage[section]{placeins}

\usepackage{url}

\usepackage{color}

\usepackage{amsmath,amssymb,mathrsfs,bbm,eufrak}
\usepackage{mathtools}
\usepackage{amsthm}

\usepackage{graphicx}
\usepackage{subfigure}

\usepackage{enumerate}

\usepackage[normalem]{ulem}

\usepackage{fancyhdr}

\theoremstyle{plain}
\newtheorem{thm}{Theorem}[section]
\newtheorem{prop}[thm]{Proposition}
\newtheorem{lem}[thm]{Lemma}

\theoremstyle{definition}

\newtheorem{defn}[thm]{Definition}

\theoremstyle{remark}
\newtheorem{rem}[thm]{Remark}

\newcommand{\R}{\mathbbm{R}}

\newcommand{\filt}[1]{\mathscr{F}_{#1}}

\renewcommand{\P}{\mathbf{P}}
\newcommand{\PP}[1]{\mathbf{P}\left(\, #1 \,\right)}

\newcommand{\EE}[1]{\mathbf{E}\left(\, #1 \,\right)}

\newcommand{\VV}[1]{\mathbf{Var}\left( \,#1 \,\right)}

\newcommand{\Cov}[2]{\mathbf{Cov}\left( #1 , #2 \right)}

\newcommand{\e}{\mathrm{e}}

\newcommand{\eps}{\varepsilon}

\begin{document}

\title{Pinned diffusions and Markov bridges}
\author{
Florian HILDEBRANDT
\footnote{Department of Mathematics, University of Hamburg, Bundesstra\ss e 55, 20146 Hamburg,
Germany}\ ,\ 
Sylvie R\OE LLY
\footnote{Institut f\"ur Mathematik der Universit\"at Potsdam,
Karl-Liebknecht-Stra\ss e 24-25, 14476 Potsdam OT Golm, Germany}
}
\date{}

\maketitle{}

\begin{abstract}
\noindent \textbf{Abstract.} In this article we consider a family of real-valued diffusion processes on the time interval $[0,1]$ indexed by their prescribed initial value $x \in \R$ and another point in space, $y \in \R$. We first  present an {\em easy-to-check} condition on their drift and diffusion coefficients ensuring that the diffusion is pinned in $y$ at time $t=1$.
Our main result then concerns the following question: can this family of pinned diffusions be obtained as the bridges either of a Gaussian Markov process or of an It\^o diffusion? 
We eventually illustrate our precise answer with several examples.
\end{abstract}


\vspace{6mm}

% classification AMS ---------------------------------------------------------
% fin classification AMS -----------

\section{Introduction}

In this article we consider for any $x,y \in \R$ the real-valued stochastic process $(X^{xy}_t)_{t \in [0,1)}$  defined as the solution to the stochastic differential equation (SDE)
\begin{equation} \label{eq:SDE}
\begin{dcases}
dX_t = h(t)(y-X_t)\, dt + \sigma(t)\,dB_t,\quad t \in [0,1)\\
X_0 = x
\end{dcases}
\end{equation}
where  $h$ and $\sigma$ are time functions and $B$ is a standard Brownian motion. \\
We are concerned in particular with the situation where the diffusion is degenerated at time $1$ in the sense that 
$ \lim_{t \to 1} X^{xy}_t= y \quad \P $-a.s. for any $y\in \R$.
In such a case, we say that the diffusion is {\em pinned}.\\

A celebrated basic example of such a family of diffusions is given by the Brownian bridges, obtained by setting $h(t) = \frac{1}{1-t},\, t \in [0,1)$ and $\sigma \equiv 1$.
Another more sophisticated example studied during the last decade  concerns the function $h(t) = \frac{\alpha}{1-t},\,t\in [0,1),$ for some $\alpha>0$. The resulting process is called \emph{$\alpha $-Wiener bridge}, \emph{scaled Wiener bridge} or \emph{$\alpha$-Brownian bridge} in the literature, see e.g.~\cite{Mansuy04}, \cite{Barczy10}. 
 For reasons we will explain in Section 4 we prefer to introduce the new terminology \emph{$\alpha$-pinned Brownian diffusion} for these processes. 
In  \cite{Li16}, considering the case of an arbitrary $h$ but  $\sigma \equiv 1$,  the author calls them \emph{generalized Brownian bridges} and studies their associated Cameron-Martin spaces. A generalization to manifold-valued processes is treated in \cite{Li17}.\\

In Section 2 of this paper we present a set of simple (non-)integrability conditions (A1)-(A2) on the functions $h$ and $\sigma$ that ensure that the solution of \eqref{eq:SDE} is pinned in $y$ at time $t=1$. 
Our criteria are easy to check and complement these of \cite{BarczyKern14}, see Remark \ref{rem:BK}. In \cite{Li16}, the case $\sigma \equiv 1$ is treated via an appropriate approximation of the identity.\\

Brownian bridges were originally obtained by {\em pinning} a Brownian motion at initial and terminal time. This means, using its Markovian characterization, the (well known) transition density of a Brownian bridge between $x$ and $y$ is given for any  $0\leq s <t<1 $ and $ u,v\in \R$ by  
\mbox{ $p^y(s,u;t,v)= \displaystyle \frac{p(s,u;t,v)p(t,v;1,y)}{p(s,u;1,y)}$}
where $p$ is the  transition density of Brownian motion.
Using this property as definition for the bridges of a Markov process, it is natural to address the following question in our much more general context:

\begin{quote}
Consider the family of pinned diffusions $\{X^{xy},\, x,y \in \R\}$ solving \eqref{eq:SDE} with two given functions $h$ and $\sigma$. Is it possible to find a continuous Markov process $Z$ whose family of bridges coincides with $\{X^{xy},\, x,y \in \R\}$?
\end{quote}

In Section 3 we present a complete answer while imposing on $Z$ to belong to specific classes of processes. First we suppose in Theorem \ref{thm:GM} that $Z$ is a Gaussian process and use as tool the fact that each centered continuous Gaussian Markov process can be represented as a space-time scaled Brownian motion. Then, we treat in Theorem \ref{thm:Ito} the complementary diffusion setting supposing that $Z$ is an It\^o diffusion satisfying 
$$
dZ_t=  b(t,Z_t)\,dt+\rho(t,Z_t)\,dB_t
$$
 where $b$ and $\rho$ are smooth functions. In this framework, our method relies on the computation of the so-called \emph{reciprocal characteristics} $(F_{b,\rho},\rho^2)$ associated to $Z$. These two space-time functions are a well adapted tool since they are invariant inside (and characterize in some sense) the whole family of bridges of $Z$ (see Proposition \ref{thm:Clark} for more details).\\
Some related problems have been partially studied in the literature, but only in a Gaussian framework  and only for particular pinned Brownian diffusions: Mansuy proved in \cite{Mansuy04} that the $\alpha$-pinned Brownian diffusion from $x=0$ to $y=0$ is not the 0-0-bridge of a time-homogeneous Gaussian Markov process. Barczy and Kern studied a similar question for a general $\alpha$-pinned Brownian diffusion from 0 to 0, see \cite{BarczyKern11}.\\

In Section 4 we eventually  treat a number of examples of pinned diffusions with the methods developed in this article. These examples include $\alpha$-pinned Brownian diffusions and $F$-Wiener bridges.\\
The consideration of similar questions for multidimensional diffusions is much more trickier since the set of reciprocal characteristics associated with a vector-valued It\^o-diffusion becomes much larger: it contains space-time functions describing the failure of the drift being in gradient form. We postpone this study to a forthcoming paper.

%%%%%%%%%%%%%%%%%%%%%%%%%%%%%%%%%%%%%%%%%%%%%%%%%%%%%%
\section{The framework: Families of pinned diffusions}
%%%%%%%%%%%%%%%%%%%%%%%%%%%%%%%%%%%%%%%%%%%%%%%%%%%%%%%
%%%%%%%%%%%%%%%%%%%%%%%%%%%%%%%%%%%%%%%%%%%%%%%%%%%%%%%

Throughout, we work on a filtered probability space $(\Omega,\filt{},(\filt{t})_{t\in [0,1]},\P)$ satisfying the usual conditions.
$B$ denotes a standard Brownian motion with respect to $(\filt{t})$. 
In the whole article we suppose that $h$ and $\sigma$ are two continuous functions defined on $[0,1)$ and that $\sigma$ is positive. These conditions are satisfied by all of the basic examples recalled in the introduction.\\

We consider the family of processes $\{X^{xy},\, x,y \in \R\}$ where $(X^{xy}_t)_{t \in [0,1)}$ is  the unique strong solution to the SDE \eqref{eq:SDE}. It is straightforward to verify that 
\begin{equation} \label{eq:sol}
X^{xy}_t = \phi(t) x+ (1-\phi(t) )y+\phi(t) \int_0^t \frac{\sigma(r)}{\phi(r)}\, dB_r \,,\quad t \in [0,1),
\end{equation}
where
\begin{equation} \label{eq:phi}
\phi(t) := \exp \Big( - \int_0^t h(r)\,dr \Big)
\end{equation}
is a 
function of class $\mathcal{C}^1$ on $[0,1)$.
Since the integrand in the stochastic integral in \eqref{eq:sol} is deterministic and locally bounded, $X^{xy}$ is a Gaussian process with first and second moments
\begin{align}
\EE{X^{xy}_t} &= \phi(t)\, x + (1-\phi(t))\,y,\quad t \in [0,1), \label{eq:exp}\\
\Cov{X^{xy}_s}{X^{xy}_t} &= \phi(s) \, \phi(t) \int_0^s \frac{\sigma^2(r)}{\phi^2(r)}\,dr \, ,\quad 0\leq s\leq t<1.\label{eq:cov}
\end{align}

We are interested in the following property.

\begin{defn}
The family of processes $\{X^{xy},\, x,y \in \R\}$ given by \eqref{eq:sol} is called a \emph{family of pinned diffusions} if for all $x,y \in \R$,
\begin{equation} \label{eq:pinn}
\PP{\lim_{t \to 1}X_t^{xy} = y}=1 .
\end{equation}
\end{defn}

\subsection{A simple condition ensuring pinning}
We now show that the following additional assumptions on $h$ and $\sigma$ ensure the pinning property \eqref{eq:pinn}.
\begin{itemize}
\item[(A1)] $\qquad h$ is bounded from below and $ \quad\displaystyle \lim_{t\to 1} \int_0^t h(r)\,dr = + \infty.$
\item[(A2)] $\qquad \lim_{t\to 1}\int_0^t\sigma^2(r)\,dr < +\infty $. 
\end{itemize}

\begin{prop} \label{thm:pinn}
If Assumptions {\em (A1)-(A2)} are satisfied then the equation \eqref{eq:sol} defines a family of pinned diffusions.
\end{prop}

\begin{proof}
Due to Assumption (A1) the following assertion holds: $\lim_{t\to 1} \phi(t)= 0$.
Therefore it only remains to show that
\begin{equation} \label{eq:ToShow-1}
\lim_{t\to 1}\phi(t) \, M_t = 0 \quad\text{a.s.}
\end{equation}
where $M_t := \displaystyle \int_0^t \frac{\sigma(r)}{\phi(r)}\,dB_r,\, t \in [0,1)$. 
We adapt the argument used in \cite[Proposition 1]{Li16} in a slightly simpler framework, where the function $\sigma$ is indeed constant.
Define 
$N_t:=\int_0^t \sigma(r)\,dB_r,\, t \in [0,1),$
which --- in view of assumption (A2) --- extends to a 
continuous $L^2$-martingale on $[0,1]$. 
The integration by parts formula for the product of continuous semimartingales leads to 
$$
d\left( \phi(t)^{-1}N_t \right)= dM_t - N_t \frac{\phi'(t)}{\phi^2(t)}\,dt.
$$ 
Solving for $M_t$ and integrating yields
\begin{align*}
M_t &= \frac{N_t}{\phi(t)} + \int_0^t N_r \, \frac{\phi'(r)}{\phi^2(r)}\, dr\\
& =\frac{N_t}{\phi(t)} -\int_0^t N_r \,  \frac{h(r)}{\phi(r)} \, dr
\end{align*} 
and therefore, 
\begin{equation*} \label{eq:phiM}
\phi(t)M_t=N_t-\int_0^t N_r \, G_t(r)\,dr,
\end{equation*}
where $G_t(r):= \displaystyle h(r) \frac{\phi(t)}{\phi(r)}= h(r)\exp\left(-\int_r^t h(s)\,ds \right)$. 

Now, fix $0<t_0<1$ and consider $t\in (t_0,1)$.
\begin{align} \label{equ:8}
\left|N_t-\int_0^t N_r \, G_t(r)\,dr\right| 
&\leq 
|N_t| | 1 - \int^t_{t_0} G_t(r)\,dr | \nonumber\\
& +
\left|\int_{t_0}^{t} (N_t - N_r) \, G_t(r)\,dr \right| 
+ \left| \int_0^{t_0} N_r \, G_t(r)\,dr \right|.
\end{align}
Since
$\displaystyle \int_{t_0}^{t}G_t(r)\,dr= \displaystyle  1-\frac{\phi(t)}{\phi(t_0)}$
the first term vanishes when $t$ tends to 1.
To treat the second term take $m\geq 0$ such that $-m\leq h(t)$ for all $t\in [0,1)$. Then $|h|\leq m+h$ and therefore,  
\begin{align*}
\int_{t_0}^t|G_t(r)|\,dr&\leq m\int_{t_0}^t \exp\left( -\int_r^t h(s)\,ds\right)\,dr+\int_{t_0}^tG_t(r)\,dr\\
&\leq m(t-t_0)\e^m+1-\frac{\phi(t)}{\phi(t_0)}
\end{align*}
which remains bounded for $t \to 1$. 
Now, it follows from the uniform continuity of $t\mapsto N_t$ on {$[0,1]$} that the second term in \eqref{equ:8} can be made arbitrarily small, as soon as $t_0$ is chosen close to 1.
The last term is bounded by 
$\displaystyle \phi(t)\int_0^{t_0}\frac{|N_r \,  h(r)|}{\phi(r)}\,dr $ which vanishes when $t$ tends to 1.
This completes the proof of \eqref{eq:ToShow-1}.
\end{proof}

\begin{rem} \label{rem:hnopositive}
\begin{itemize}
\item
The assumption that $h$ is bounded from below is used in the above proof to control the second term in the upper bound \eqref{equ:8}. Indeed, to suppose only 
boundedness of  $\displaystyle \int_{t_0}^t |G_t|(r)dr= \displaystyle \int_{t_0}^t |h(r)|\exp\left(-\int_r^t h(s)\,ds \right)dr$ for $t\in [t_0,1)$
would be enough to conclude, but this condition could be difficult to check. 
\item
The divergence of the integral in Assumption (A1) is indeed a necessary condition to ensure that the diffusion is pinned: Recall that convergence in probability for Gaussian random variables implies $L^1$-convergence. Hence, due to \eqref{eq:exp}, to ensure \eqref{eq:pinn} it is necessary that $\lim_{t \to 1} \phi(t)=0$ holds. The latter is equivalent to $\lim_{t\to 1} \int_0^t h(r)\,dr=+\infty$.
\end{itemize}
\end{rem}

\vspace{3mm}

\begin{rem} \label{rem:BK}
Let us compare our result with the one from Barczy and Kern \cite[Proposition 2.4]{BarczyKern14}.
They show that the diffusion $X^{xy}$ is pinned if one replaces our assumption (A2) by the following condition on $h$ and $\sigma$:
\begin{itemize}
\item[(A2')] There exists $\delta \in (0,1)$ such that 
$$
 \sup_{t\in[0,1)} \displaystyle \e^{-2\delta \int_0^t h(s)\,ds } \int_0^t \e^{2 \int_0^r h(s)\,ds} \sigma^2(r)\,dr < + \infty .
$$
\end{itemize}
Clearly, for bounded $\sigma$, Assumption (A2) is trivial to check unlike condition (A2'). 
Furthermore, there are situations in which Assumption (A2) is satisfied while (A2') does not hold: Take $\sigma \equiv 1$ and $h(t):= \frac{2-t}{(1-t)^2},\ t \in [0,1).$ 
Assumption (A1) is satisfied since $\int_0^th(s)\,ds = \frac{1}{1-t}+\log \frac{1}{1-t}-1\overset{t \to 1}{\longrightarrow} + \infty.$ Assumption (A2) is also satisfied.
Therefore, we are dealing with a family of pinned diffusions. 
On the other hand, for  $t \in [0,1)$, 
\begin{multline*}
\int_0^t \e^{2 \int_0^r h(s)\,ds}\,dr 
=\int_0^t \e^{\frac{2}{1-r}-2}\frac{1}{(1-r)^2}\,dr
=\int_1^{\frac{1}{1-t}} \e^{2r-2}\,dr= \frac{1}{2}\left(\e^\frac{2t}{1-t}-1\right). 
\end{multline*}
For any $\delta \in (0,1)$, it follows that
$$
\displaystyle  \e^{-2\delta \int_0^t h(s)\,ds } \int_0^t \e^{2 \int_0^r h(s)\,ds} \,dr =
\frac{1}{2}(1-t)^{2\delta}\e^{-\frac{2\delta t}{1-t}}
\left(\e^\frac{2t}{1-t}-1\right) \overset{t \to 1}{\longrightarrow}+\infty.
$$
Consequently (A2') is not satisfied.
\end{rem}

%%%%%%%%%%%%%%%%%%%%%%%%%%%%%%%%%%%%%%%%%%%%%
\section{Identification of bridge processes}
%%%%%%%%%%%%%%%%%%%%%%%%%%%%%%%%%%%%%%%%%%%%
%%%%%%%%%%%%%%%%%%%%%%%%%%%%%%%%%%%%%%%%%%%%
In this section we examine for which choices of $h$ and $\sigma$ a family of pinned diffusions $\{X^{xy},\,x,y\in \R\}$ can be obtained as the family of bridges of a single Markov process $Z$. The following theorem answers this question if $Z$ is assumed to be Gaussian. 
\begin{thm} \label{thm:GM}
Assume that the processes $\{X^{xy},\, x,y\in \R\}$ defined by \eqref{eq:sol} constitute a family of pinned diffusions (e.g.~if (A1)-(A2) are satisfied). They correspond to the bridges of a non-degenerate Gaussian Markov process $Z$ if and only if 
\begin{equation} \label{eq:Sigma}
\Sigma:= \lim_{t \to 1}\int_0^t \sigma^2(r)\,dr < + \infty
\end{equation}
and the function $h$ is related to $\sigma$ as follows:
\begin{equation} \label{eq:hFbridge}
h(t)= \frac{\sigma^2(t)}{\Sigma- \int_0^t \sigma^2(r)\,dr},\quad t\in [0,1).
\end{equation}
In this case the original process $Z$ follows the dynamic $dZ_t = \sigma(t)\,dB_t$.
\end{thm}

For the proof we need the following representation of continuous Gaussian Markov processes.

\begin{lem} \label{lem:GM}
Let $Z= (Z_t)_{t\in [a,b]}$ be a non-degenerate continuous Gaussian Markov process. %satisfying $\VV{X_t}>0$ for each $t \in [a,b]$. 
Then there exist three continuous functions $m,u,v:[a,b]\to \R$ such that 
$$
Z\overset{(d)}{=} m(\cdot)+u(\cdot)\hat B_{v(\cdot)}  ,
$$
where $u$ and $v$ are strictly positive, $v$ is non-decreasing and $\hat B$ is a Brownian motion.
\end{lem}
\begin{proof}
Since a continuous Gaussian process is continuous in $L^p$ for any $p \geq 1$ the functions $t \mapsto \EE{Z_t}$ and $(s,t)\mapsto \Cov{Z_s}{Z_t}$ are continuous. Define $m$ as the first moment: $m(t):= \EE{Z_t},\,t \in [a,b]$. Then $Z-m$ is a centered and non-degenerate Gaussian Markov process having a continuous covariance function. The claimed representation now follows from \cite[\S 3.2]{Neveu68}.
\end{proof}

\begin{proof}[Proof of Theorem \ref{thm:GM}]
Suppose first there exists a non-degenerate continuous Gaussian Markov process $Z$   whose bridges are given by $\{X^{xy},\,x,y\in \R\}$. Let  $Z^x$ be the process $Z$ started in $x$. Since we assume $Z$ to be non-degenerate, $\VV{Z_t^x}>0$ holds for all $t \in (0,1]$. By Lemma \ref{lem:GM}, for any $\eps \in (0,1)$ the following representation holds:
$$
Z^x|_{[\eps,1]} \stackrel{(d)}=m(\cdot)+u(\cdot)\hat B_{v(\cdot)}
$$
where $\hat B$ is a Brownian motion, $m,u,v:[\eps,1] \to \R$ are continuous, $u$ and $v$ are strictly positive and $v$ is non-decreasing. Since $u(\cdot)\hat B_{cv(\cdot)}$ and $\sqrt{c}u(\cdot) \hat B_{v(\cdot)}$ have the same distribution for any $c>0$ we can suppose $v(1)= 1.$  Accordingly, the covariance of $Z^x$ has the representation 
$$
c(s,t):=\Cov{Z_s^x}{Z_t^x}= u(s)u(t)v(s),\quad s \leq t.
$$
It follows that expectation and covariance of the bridge process $Z^{xy}$ (corresponding to the conditional moments given $Z_1=y$) are given by (see e.g.~ \cite[p.~12]{Muirhead82})
\begin{align}
\EE{X^{xy}_t} &= m(t)+\frac{c(t,1)}{c(1,1)}(y-m(1))\nonumber\\
&=m(t)+\frac{u(t)v(t)}{u(1)}(y-m(1)), \quad t \in (0,1], \label{eq:bridge.exp}\\
\Cov{Z^{xy}_s}{Z^{xy}_t} &= c(s,t)-\frac{c(s,1)c(t,1)}{c(1,1)}\nonumber\\
&= u(s)v(s)u(t) \left(1-v(t) \right),\quad s\leq t.\label{eq:bridge.cov}
\end{align}
Comparison with \eqref{eq:exp} and \eqref{eq:cov} yields  
\begin{align*}
u(s)v(s) = \phi(s) \int_0^s \frac{\sigma^2(r)}{\phi^2(r)}\,dr,\quad
 \frac{u(s)v(s)}{u(1)} = 1-\phi(s),\quad s \in (0,1).
\end{align*}
This implies 
$$
\int_0^s \frac{\sigma^2(r)}{\phi^2(r)}\,dr = u(1) \, \frac{1-\phi(s)}{\phi(s)},\quad s \in (0,1).
$$ 
Differentiating both sides of this equality yields $\phi' = -\frac{\sigma^2}{u(1)}$ and then
 $$
\exists C \in \R, \quad \phi(t)=C -\frac{1}{u(1)}\int_0^t \sigma^2(r)\,dr.
$$ 
Since $\phi(0)=1$ and $\lim_{t \to 1}\phi(t)= 0$ (see Remark \ref{rem:hnopositive}), this implies $C=1$ and $ \Sigma = u(1)< + \infty$.  Using \eqref{eq:phi} allows to deduce the desired expression for $h$.\\
Conversely, let $Z^x$ be defined by $Z^x_t = x + \int_0^t \sigma(s)\,dB_s,\, t \in [0,1]$, and let $h$ be given by \eqref{eq:hFbridge}. In this case we obtain as an alternative representation $Z^x_t \stackrel{(d)}{=} x + u(t)\hat{B}_{v(t)},\,t \in [0,1],$ where $u(t) \equiv \sqrt{\Sigma}$ and $v(t):= \frac{1}{\Sigma}\int_0^t \sigma^2(r)\,dr$. A direct computation yields $\phi = 1-v$  and
$$\int_0^t \frac{\sigma^2(r)}{\phi^2(r)}\,dr= \Sigma \ \frac{ v(t)}{1-v(t)},\quad t \in [0,1).$$
Therefore, the moments defined by \eqref{eq:exp} and \eqref{eq:cov}, respectively \eqref{eq:bridge.exp} and \eqref{eq:bridge.cov} agree. It remains to prove that the diffusion is pinned, i.e.~to verify that \eqref{eq:pinn} holds. 
This can be done by checking conditions (A1) and (A2). (A1) holds since $h$ is positive and $\lim_{t \to 1} \phi(t)=0$ and (A2) holds by assumption.
\end{proof}
%%%%%%%%%%%%%%%%%%%%%%%%%%
\begin{rem} \label{rem:BK13}
In the particular  case where $\sigma$ extends to a continuous function on $[0,1]$, the second part of the above proof is a consequence of Theorem 3.1 and Theorem 3.2 in \cite{BarczyKern13}.
\end{rem}

\vspace{0.7cm}

We continue looking for a process $Z$ whose bridges match the family $\{X^{xy},\,x,y\in \R\}$ but we definitely complement the class of allowed $Z$ considering now It\^{o} diffusions which are a priori not Gaussian.

\begin{defn}\label{def:regIto}
A weak solution $Z  = (Z_t)_{t \in [0,1]}$ of a SDE of the form 
$$
dZ_t = b(t,Z_t)\,dt + \rho(t,Z_t)\,dB_t,
$$
is called a {\em regular It\^{o} diffusion} if the coefficients  $b$ and $\rho$ are in $ C^{1,2}([0,1]\times\R)$, if $ \rho$ is strictly positive  and if $Z$ admits a  strictly positive  transition density $p$ such that
\begin{enumerate}[(i)]
\item $p$ is differentiable in all variables;
\item For any $(t_1,z_1)\in [0,1)\times \R$, the partial derivatives of $p(t_1,z_1;\cdot,\cdot)$ exist up to order two and are continuous.
\end{enumerate}
\end{defn}
We are now able to state  our main result.
\begin{thm} \label{thm:Ito}
Suppose that the functions $h$ and $\sigma$ are of class $C^1$ on $[0,1)$ and assume that the associated processes $\{X^{xy},\,x,y\in \R\}$ given by \eqref{eq:sol} are pinned diffusions. Then $\{X^{xy},\,x,y\in \R\}$ correspond to the bridges of a regular It\^{o} diffusion $Z$ if and only if \eqref{eq:Sigma} and \eqref{eq:hFbridge} are satisfied.
In this case $Z$ is indeed Gaussian and one can choose its drift coefficient $b\equiv 0$ and its diffusion coefficient $ \rho^{2}(t,z) \equiv \sigma^{2} (t)$ as in Theorem \ref{thm:GM}.
\end{thm}

The proof is based on the notion of  {\em reciprocal characteristics}, which was first introduced by Clark \cite{Clark90}:
For any regular It\^{o} diffusion $Z$ with coefficients $b$ and $\rho$ define the space-time function
$$
F_{b,\rho}(t,z):=\partial_t (b/\rho^2)(t,z)+\frac{1}{2}\partial_z \Big((b/\rho)^2(t,z) +\rho^2 (t,z)\partial_z (b/\rho^2)(t,z)  \Big).
$$
The function $F_{b,\rho}$ together with $\rho^2$ are called the reciprocal characteristics associated with $Z$. Clark asserts that they are invariant inside the class of It\^o diffusions which share the same bridges. For convenience of the reader we recall his precise result in the following proposition. 

\begin{prop}  \label{thm:Clark}
Let $X  = (X_t)_{t \in [0,1]}$ and $\tilde X = (\tilde X_t)_{t \in [0,1]}$ be two regular It\^{o} diffusions solving the SDEs 
\begin{align*}
dX_t &= b(t,X_t)\,dt + \rho(t,X_t)\,dB_t,\\
d\tilde X_t &=  \tilde b(t,\tilde X_t)\,dt +  \tilde \rho(t,\tilde X_t)\,d\tilde B_t.
\end{align*}
Then $X$ and $\tilde X$ share the same bridges if and only if their reciprocal characteristics coincide, that is  
\begin{equation} \label{eq:recCh}
\rho^{2} \equiv  \tilde \rho^{2} \quad \text{and}\quad F_{b,\rho}\equiv F_{ \tilde b,  \tilde \rho}.
\end{equation}
\end{prop}
\begin{proof}
For a detailed proof under the above assumptions we refer to \cite{Thieullen02}. 
To briefly summarize, the main argument leads to the fact that $X$ and $\tilde X$ share the same bridges if and only if they are $\mathfrak h$-tranforms in the sense of Doob. Equivalently, there exists a function $\mathfrak h \in C^{1,2}([0,1)\times \R )$ such that
$$\tilde \rho \equiv \rho\quad \text{ and } \quad \tilde b \equiv b + \rho^2\,\partial_z \log \mathfrak h$$ 
where $\mathfrak h$ is space-time harmonic, i.e.~satisfies
$$
\partial_t \mathfrak h + \frac{1}{2}\rho^2 \, \partial_{zz}\mathfrak h + b \, \partial_z \mathfrak h = 0.
$$
One of the calculations in this proof requires the validity of $\partial_t\partial_y p(s,x;t,y)\linebreak=\partial_y\partial_t p(s,x;t,y)$ which justifies the presence of condition (ii) in Definition \ref{def:regIto}. The differentiability condition (i) on the first two variables is necessary to ensure regularity for the bridge transition density in the second two variables.
\end{proof}

\begin{proof}[Proof of Theorem \ref{thm:Ito}]
Assume that $\{X^{xy},\,x,y\in \R\}$ defined by \eqref{eq:sol} are the bridges of a regular It\^o diffusion $Z=(Z_t)_{t\in [0,1]}$ which solves
$$
dZ_t = b(t,Z_t)\,dt + \rho(t,Z_t)\,dB_t.
$$
Define $b_y(t,z) := h(t)(y-z)$, $t\in [0,1),z \in \R$. Proposition \ref{thm:Clark} implies that  for each $y \in \R$,  the following system of equations holds on $[0,1)\times \R$:
\begin{equation*} 
\begin{dcases}
\rho\equiv \sigma \\
F_{b,\rho}\equiv F_{b_y,\sigma}.
\end{dcases}
\end{equation*}
In particular $F_{b_y,\sigma}(t,z)= \displaystyle \frac{h'(t)\sigma^2(t)-h(t) \partial_t\sigma^2(t)-h^2(t)\sigma^2(t)}{\sigma^4(t)}\, \big(y-z\big)$ should not depend on $y$. Hence, 
$$
h'=h^2+h g \quad \textrm{ where } \quad 
g = \partial_t\log \sigma^2 .
$$
 Setting $G(t):= \int_0^t g(s)\,ds= \log \frac{\sigma^2(t)}{\sigma^2(0)},\, t \in [0,1)$, we obtain as the unique solution of this ODE
 $$
h(t)= \frac{\e^{G(t)}}{C-\int_0^t \e^{G(s)}\,ds}= \frac{\sigma^2(t)}{\tilde C-\int_0^t \sigma^2(s)\,ds}
$$ on its maximal interval of existence, where $C >0$ and $ \tilde C = \sigma^2(0)C$. Now,
\begin{align*}
\int_0^t h(s)\,ds = \log\frac{\tilde C}{\tilde C - \int_0^t \sigma^2(s)\,ds}
\end{align*}
and by Remark \ref{rem:hnopositive}, $\tilde C = \Sigma< + \infty$ is uniquely determined. \\
To prove the converse implication one follows the same argumentation  as in the proof of Theorem \ref{thm:GM}.
\end{proof}

%%%%%%%%%%%%%%%%%%
\section{Examples}
%%%%%%%%%%%%%%%%%%
%%%%%%%%%%%%%%%%%%

%%%%%%%%%%%%%%%%%%%%%%%%%%%%%%%%%%%%%%%%%%%%%%%%
\subsection{$\alpha$-pinned Brownian diffusions}
%%%%%%%%%%%%%%%%%%%%%%%%%%%%%%%%%%%%%%%%%%%%%%%%
We first consider the family $\{X^{xy},\,x,y\in \R\}$ of pinned diffusions associated with
\begin{equation*}
\sigma \equiv 1 \qquad \text{ and } \qquad h(t) = \frac{\alpha(t)}{1-t},\quad t \in [0,1),
\end{equation*}
where $\alpha : [0,1] \to \R$ is continuous and satisfies $\alpha(1)>0$. They solve the SDE
\begin{equation} \label{eq:alphaWiener}
\begin{dcases}
dX_t = \alpha(t)\,\frac{y-X_t}{1-t}\,dt +dB_t,\quad t \in [0,1),\\
X_0 =x.
\end{dcases}
\end{equation}
 In the particular case  $x=y=0$, they  were introduced by Barczy and Kern under the name of \emph{general $\alpha$-Wiener bridges}. They generalize  the so-called $\alpha$-Wiener bridges ($\alpha(t) \equiv \alpha(0)$) which were first introduced by Brennan and Schwartz \cite{Brennan90} to model the arbitrage profit associated to a given stock index future in absence of transaction costs. For a mathematical treatment of $\alpha$-Wiener bridges see also Mansuy \cite{Mansuy04}, Barczy and Pap \cite{Barczy10} and Li \cite{Li16}. \\
Barczy and Kern proved in \cite{BarczyKern11} that the process $X^{00}$ given by \eqref{eq:alphaWiener} is pinned  in 0 at time $t = 1$. They then studied whether $X^{00}$ can be obtained as the 0-0-bridge of an Ornstein-Uhlenbeck type process, i.e.~a process $Z$ satisfying $dZ_t = q(t)\,dt +\sigma(t)\,dB_t$ for continuous functions $q,\sigma:[0,1]\to \R$ where $\sigma(t)\neq 0$ for all $t \in [0,1]$. 
In Theorem 4.1 of \cite{Barczy10},  Barczy and Pap showed that for  $\alpha\neq \alpha'$ constant in time, the \textit{generalized Brownian bridges with parameters} $\alpha$ and $\alpha '$ have  singular laws. Also, refer to Example 4 of \cite{Li16} for a treatment of this question with a different approach.

\vspace{2mm}
Assumption (A2) is trivially satisfied. Let us show that Assumption (A1) is satisfied.
Since $\alpha$ is continuous on $[0,1]$ and $\alpha(1)$ is positive,  there exists $t_0 \in (0,1)$ such that $\alpha$ is bounded away from 0 on $[t_0,1]$: $$
\exists \eps>0, \, \forall t \in [t_0,1], \quad   \alpha(t)\geq \epsilon.
$$ 
This implies that for any $t \in [t_0,1)$,
$$ 
\int_0^t h(u)\,du \geq\int_0^{t_0} h(u)\,du + \eps \int_{t_0}^t \frac{du}{1-u}\overset{t \to 1}{\longrightarrow} \infty,
$$
yielding (A1). 
Thus, due to  Proposition \ref{thm:pinn}, the pinning property of the solution $X^{xy}$ of \eqref{eq:alphaWiener} is assured for any $x,y \in \R$.\\
Further, as a corollary of Theorems \ref{thm:GM} and \ref{thm:Ito}, one deduces that, if the function $\alpha$ differs from the  constant 1, the family $\{X^{xy},\, x,y \in \R\}$ of such pinned diffusions does not coincide with the bridges of any Gaussian Markov process or of any regular It\^{o} diffusion. \\
For this reason we propose to call the solution of \eqref{eq:alphaWiener} an \emph{$\alpha$-pinned Brownian diffusion} without using the word \emph{bridges} which is not well adapted to this situation.

%%%%%%%%%%%%%%%%%%%%%%%%%%%%%%%%%%%%%%%%%%%%%%%%%%%%%%%%%
\subsection{$(\alpha,\gamma)$-pinned Brownian diffusions}
%%%%%%%%%%%%%%%%%%%%%%%%%%%%%%%%%%%%%%%%%%%%%%%%%%%%%%%%%
As a modification of the previous example we consider for $\alpha > 0$ and $\gamma \geq 0$ the following SDE	
\begin{equation}
\begin{dcases}
dX_t = \alpha\,\frac{y-X_t}{(1-t)^{1+\gamma}}\,dt + dB_t,\quad t \in [0,1),\\
X_0 = x,
\end{dcases}
\end{equation}
and its solution $X^{xy}$. 
The function 
$h(t)= \alpha /(1-t)^{1+\gamma}$ clearly satisfies 
Assumption (A1). 
It thus follows from Proposition \ref{thm:pinn} that $\{X^{xy},\,x,y\in \R\}$ are pinned diffusions and 
 it makes sense to call them \emph{$(\alpha,\gamma)$-pinned Brownian diffusions}. \\
Theorem \ref{thm:GM} applies and yields that for $(\alpha,\gamma)\neq (1,0)$ the $(\alpha,\gamma)$-pinned Brownian diffusions cannot be obtained as the bridges of any Gaussian Markov process.
Theorem \ref{thm:Ito} applies, too: $\{X^{xy},\, x,y \in \R\}$ can not be obtained as the bridges of any regular It\^{o} diffusion as soon as we are not in the very particular case $(\alpha,\gamma)= (1,0)$.\\
In \cite{Li16} Example 2, the author shows that for $x=y=0$ and $\alpha=1$, this \textit{generalized Brownian bridge} has a law which is singular with respect to the law of the Brownian bridge as soon as $\gamma>0$.

%%%%%%%%%%%%%%%%%%%%%%%%%%%%%%%
\subsection{$F$-Wiener bridges}
%%%%%%%%%%%%%%%%%%%%%%%%%%%%%%%
Let $f:[0,1)\to (0,+\infty)$ be a continuous probability density function and let $F$ denote the corresponding cumulative distribution function. We now consider the SDE
\begin{equation} \label{FWiener}
\begin{dcases}
dX_t = \frac{f(t)}{1-F(t)}(y-X_t)\,dt + \sqrt{f(t)}\,dB_t,\quad t \in [0,1)\\
X_0 =x.
\end{dcases}
\end{equation}
and its solution $X^{xy}$. \\
In the particular case $x=y=0$ these processes are called $F$-Wiener bridges and play a role in statistics as weak limits of empirical processes, see e.g.~\cite{Vaart2000}.
It was shown in \cite{BarczyKern14} that $X^{00}$ is pinned in 0 at time $t = 1$.\\
We are indeed in the framework of Theorem \ref{thm:GM}: Assumptions (A1) and (A2) are satisfied with $\sigma= \sqrt{f}$ and $\Sigma= 1$.
Thus the family $\{X^{xy},\, x,y \in \R\}$ coincides with the bridges of the Gaussian Markov process 
$$
Z_t = \int_0^t \sqrt{f(s)}\,dB_s \overset{(d)}{=} B_{F(t)},\quad t \in [0,1].
$$
%which can be regarded both as a Gaussian process and a regular It\^{o} diffusion.

\subsection*{Acknowledgement} The authors thank an anonymous referee for bringing to their attention the references \cite{Li16} and \cite{Li17}.

\bibliography{ref_PinnedDif}
\bibliographystyle{ieeetr}
\end{document}